\documentclass[journal, twoside]{IEEEtran}

\usepackage{amsmath,amssymb,amsthm}
\DeclareMathOperator{\sinc}{\mathrm{sinc}}
\DeclareMathOperator{\proj}{\mathsf{proj}}
\DeclareMathOperator{\lspan}{\mathsf{span}}
\newcommand{\ideal}[1]{{\langle{#1}\rangle}}

\usepackage{graphicx}
\usepackage{caption, subcaption}

\usepackage[ruled, algo2e]{algorithm2e}
\usepackage{algorithm}
\SetKwRepeat{Do}{do}{while}

\allowdisplaybreaks
\newtheorem{claim}{Claim}
\newtheorem{corollary}{Corollary}
\newtheorem{definition}{Definition}

\title{Quasicyclic Principal Component Analysis}
\author{Susanna E.~Rumsey \IEEEmembership{Graduate Student Member, IEEE}, Stark C.~Draper \IEEEmembership{Senior Member, IEEE}, and Frank~R.~Kschischang \IEEEmembership{Fellow, IEEE}\thanks{The authors are with the Edward S.~Rogers Sr.~Department of Electrical and Computer Engineering, University of Toronto.}}

\begin{document}
\maketitle

\begin{abstract}
We present quasicyclic principal component analysis (QPCA), a generalization of principal component analysis (PCA), that determines an optimized basis for a dataset in terms of families of shift-orthogonal principal vectors.  This is of particular interest when analyzing cyclostationary data, whose cyclic structure is not exploited by the standard PCA algorithm.  We first formulate QPCA as an optimization problem, which we show may be decomposed into a series of PCA problems in the frequency domain.  We then formalize our solution as an explicit algorithm and analyze its computational complexity.  Finally, we provide some examples of applications of QPCA to cyclostationary signal processing data, including an investigation of carrier pulse recovery, a presentation of methods for estimating an unknown oversampling rate, and a discussion of an appropriate approach for pre-processing data with a non-integer oversampling rate in order to better apply the QPCA algorithm.
\end{abstract}
\begin{IEEEkeywords}
Principal Component Analysis, PCA, signal processing, cyclostationarity
\end{IEEEkeywords}

\section{Introduction}

\IEEEPARstart{P}{rincipal} component analysis (PCA) is widely used throughout data science and statistics.  It provides the user with a set of orthogonal axes --- ``principal components" --- that best capture the energy of a data set.  

PCA is often used when approximating higher-dimensional data.  PCA components are ordered by the successive amounts of data energy they capture.  If we approximate the data by its projection onto the first $k$ PCA components, then, in terms of energy captured, we obtain an optimal $k$-dimensional approximation of the data.

Standard PCA does not take into account any knowledge of structure of the data.  Various generalizations and variations of PCA have been developed to address certain types of data.  These include singular spectrum analysis and related techniques for time-series data, both stationary~\cite[chap.~12]{jolliffe02} and non-stationary~\cite{hamilton24} as well as high-frequency data~\cite{aitsahalia15}.  Takane~\cite{takane13} provides an extensive description of constrained PCA, which allows the user to decompose the data matrix using external information about the data, and then create a solutions from PCA performed on the components of the decomposition.  Dai and Müller extend PCA to the case where the data lie on a general Riemannian manifolds (including the particular example of data on a sphere)~\cite{dai17}.

This paper introduces a generalization of PCA, which we name quasicyclic PCA (QPCA).  In QPCA, instead of seeking a set of vectors that best capture the energy of the experimental data, we seek families of shift-orthogonal vectors that best capture the energy.

Our approach is useful in experimental applications where the user has reason to believe that the underlying data is cyclostationary, i.e., that the statistics of each period of the data are the same.  We are particularly interested in applications to communications engineering, as information signals created by modulating symbols using shift-orthogonal pulses have such cyclostationary statistics.  However, cyclostationary data appears in many other fields, including dynamics, econometrics, astronomy, and biology.  For an extensive bibliography on this subject, see ~\cite[chap.~10]{napolitano20}.

We borrow the term ``quasicyclic" from algebraic coding theory to describe our approach to applying PCA to cyclostationary data.  In a quasicyclic code, every circular shift of a codeword by a multiple of some number of indices $s$ is also a codeword.  (See, e.g.~\cite{chen69} for an extensive discussion.)  In much the same way, as our data is cyclostationary, we consider circular shifts of our data vectors by some amount $s$ also to be relevant data.

As an example, consider Figure~\ref{fig:intro_ex}.  Here, we use the length 54 pulse shown in Figure~\ref{fig:pulse} to modulate $m=100$ data signals, each consisting of $N=6$ PAM-4 symbols (Figure~\ref{fig:data} gives an example).  This pulse is orthogonal to its circular shifts by multiples of 9 samples.  Performing PCA on this data and extracting the $N=6$ most significant orthogonal components (see Figure~\ref{fig:pca}) does not tell us much about the structure of the underlying signals.  However, QPCA recovers a size-$N$ family of orthogonal 9-shifts of a pulse, each of which closely approximates a shift of the original pulse (see Figure~\ref{fig:qpca}).

\begin{figure}[!t]
	\centering
	\begin{subfigure}[b]{0.4\textwidth}
		\centering
		\includegraphics[width=\textwidth]{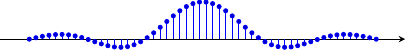}
		\caption{A pulse $\phi$ of length 54.  It is orthogonal to any $\tilde{\phi}$ formed by circularly shifting $\phi$ by any multiple of 9 time-steps (``9-shifts").}
		\label{fig:pulse}
	\end{subfigure}
	\hfill
	\begin{subfigure}[b]{0.4\textwidth}
		\centering
		\includegraphics[width=\textwidth]{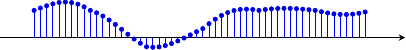}
		\caption{Example of data formed from a linear combination of $N=6$ circular shifts of $\phi$.  Our data set consists of $m=100$ such curves.}
		\label{fig:data}
	\end{subfigure}
	\hfill
	\begin{subfigure}[b]{0.4\textwidth}
		\centering
		\includegraphics[width=\textwidth]{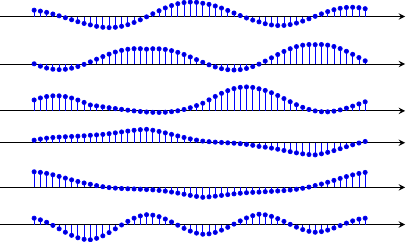}
		\caption{The first $N=6$ components obtained by performing regular PCA on our data.}
		\label{fig:pca}
	\end{subfigure}
	\hfill
	\begin{subfigure}[b]{0.4\textwidth}
		\centering
		\includegraphics[width=\textwidth]{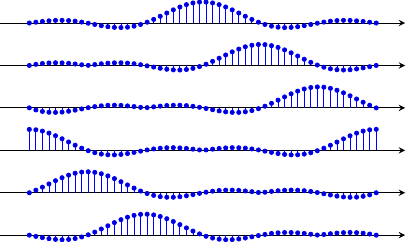}
		\caption{The first component and its $N=6$-dimensional family of orthogonal shifts obtained by performing QPCA on our data.  Each is close to a 9-shift of $\phi$.}
		\label{fig:qpca}
	\end{subfigure}
	\caption{An illustration of QPCA.  In~(a), a shift-orthogonal pulse $\phi$ is used to create data consisting of $n=100$ runs of signals containing $N=6$ modulated symbols each, such as that shown in~(b).  The first $N$ orthogonal PCA components are shown in~(c), and have a complicated form that tells the user little about the structure of the data.  In contrast, the QPCA results in~(d) give the first shift-orthogonal component (along with its orthogonal family of shifts), which accurately corresponds to the underlying pulse $\phi$.}
	\label{fig:intro_ex}
\end{figure}

This simple example illustrates the potential usefulness of QPCA, which will be fully described and formalized in the sections below.  The rest of the paper is structured as follows: Section~\ref{sec:PCA} provides a mathematical description of standard PCA as background.  Section~\ref{sec:problem} describes the formulation of QPCA as an optimization problem.  Section~\ref{sec:solution} derives the solution to the problem.  Finally, Section~\ref{sec:apps} provides some examples uses of QPCA.

We will use the following notation throughout: signals (also referred to as vectors) are elements of $\mathbb{C}^n$, with indices starting at 0.  They are indexed modulo $n$, i.e., for a signal $x$, we understand $x(i)$ to mean the $i\bmod n$th element of $x$.  We denote the standard inner product of signals $x$ and $y$ by $\langle x,y \rangle = \sum_{i = 0}^{n - 1} x(i) y^*(i)$, and the norm of $x$ by $\|x\| = \sqrt{\langle x, x\rangle}$.  Products of signals are understood to be taken pointwise, i.e., $(xy)(i) = x(i)y(i)$.

\section{Principal Component Analysis}
\label{sec:PCA}

\noindent Suppose that we are given a set of $m$ complex signals $\{ x_1, \ldots, x_m \} \subseteq \mathbb{C}^n$, which we refer to as data vectors.  The \emph{centroid} of these data vectors is the vector $\bar{x} = \frac{1}{m} \sum_{i=1}^m x_i$. The \emph{centred data} is then $\{y_1, \dots, y_m\} = \{ x_1 - \bar{x}, \ldots, x_m - \bar{x} \}$.  We suppose that the centred data span a $d$-dimensional subspace of $\mathbb{C}^n$.  We assume $d>0$;  in most applications $d = n$.

A first \emph{principal component} of the centred data is a vector $q^{(1)} \in \mathbb{C}^n$ satisfying
\begin{equation}\label{eq:pca}
\phantom{.}q^{(1)} = \arg \max_{\| q \| =1 } \sum_{i=1}^m | \langle y_i, q\rangle |^2.
\end{equation}
The quantity $\sum_{i=1}^m | \langle y_i, q \rangle |^2$ is the sum of the energy of the signals obtained by projecting each of the centred data vectors onto~$q$.

Define the \emph{projection} of signal $x$ onto the space spanned by signal $v$ by $\proj_v(x) = \frac{ \langle x,v \rangle}{\| v \|^2} v$.  This simplifies when $\| v \| = 1$ to $\proj_v(x) = \langle x,v \rangle v$.  The $m$ error signals $y_i - \proj_{q^{(1)}}(y_i)$ are orthogonal to $q^{(1)}$ for each  $i \in \{ 1, \ldots, m \}$. A \emph{second principal component} $q^{(2)}$ is obtained as a first principal component of this collection of error signals.  Proceeding in this way, we may define, for any positive integer $i \leq d$, an $i$th principal component $q^{(i)}$ as a first principal component of
the $m$ signals
\[
\left\{
\phantom{,}y_1 - \sum_{j=1}^{i-1} \proj_{q^{(j)}} (y_1),
\;\dots, y_m - \sum_{j=1}^{i-1} \proj_{q^{(j)}} (y_m)
\right\},
\]
where an empty sum of vectors is taken to be $\mathbf{0}$.  By construction, $\{ q^{(1)}, \ldots, q^{(d)} \}$ is an orthonormal basis for the vector space spanned by the centred data vectors.

If we create an $m\times n$ matrix $\mathbf{Y}$ whose rows are the $y_i$s, it can be shown that $q^{(i)}$ is a normalized eigenvector of $\mathbf{Y}^H\mathbf{Y}$ corresponding to the $i$th largest eigenvalue.  (This is equivalent to the normalized right-singular vector of $\mathbf{Y}$ corresponding to the $i$th largest singular value of $\mathbf{Y}$.)  This so-called \emph{principal component analysis} (PCA) was developed in 1901 by Pearson~\cite{pearson01}; it was later named PCA by Hotelling in the 1930s~\cite{hotelling33}.

\section{Problem Statement}
\label{sec:problem}

\subsection{Mathematical Background}

\noindent We begin by formalizing the concept of a shift-orthonormal signal, the class of signal to which we constrain QPCA components.

\begin{definition}[Autocorrelation]\label{def:autocorr}
	The \emph{autocorrelation function} $R_x \in \mathbb{C}^n$ associated with a signal $x \in \mathbb{C}^n$ is defined pointwise for $0 \le j < n$ as
	\begin{equation}\label{eq:autocorr}
	R_x(j) = \langle x, x \circledast e_j \rangle
	= \sum_{i = 0}^{n - 1} x(i) x^*(i-j)
	\end{equation}
	where $e_j$ is the $j$th unit vector in the standard basis and $\circledast$ indicates circular convolution.
\end{definition}

Note that $x \circledast e_j$ is the circular shift of $x$ by $j$ indices, which means that the autocorrelation function returns the inner product between $x$ and the circular shift by $j$ of $x$.

\begin{definition}[Shift-orthonormal signals]\label{def:shiftorth}
	Let $s$ and $N$ be integers such that $n=Ns$. A signal $x \in \mathbb{C}^n$ is said to be \emph{$s$-quasicyclic shift-orthonormal} or an \emph{$s$-root-Nyquist pulse} if
	\[
	R_x(j) = \begin{cases}
		1, & j = 0 , \\
		0, & j = ks, \;k \ne 0.
	\end{cases}
	\]
	Such a signal is orthogonal to cyclic shifts of itself by any nonzero multiple of $s$.
\end{definition}

Let us define the Kronecker delta function as follows.

\begin{definition}
The length-$n$ Kronecker delta signal $\delta \in \mathbb{C}^n$ is defined by
\[
\phantom{.}\delta(i) = \begin{cases}
	1 & i = 0\\
	0 & i \ne 0
\end{cases}.
\]  
\end{definition}
This means that~(\ref{eq:autocorr}) can be rewritten as 
\begin{equation}\label{eq:rx}
\phantom{.}R_x \sum_{i = 0}^{N - 1}\delta\circledast e_{is} = \delta.
\end{equation}
for $0 \le j < n$.

We denote by $Q^n(s) \subseteq \mathbb{C}^n$ the set of all $s$-quasicyclic shift-orthonormal length-$n$ signals.  For any signal $q \in Q^n(s)$, we let
$\lspan^s(q) = \lspan\{q\circledast e_{ks}\}_k$ denote the vector space spanned by $q$ and its quasicyclic shifts by multiples of $s$.

We generalize the concept of projecting a signal onto a one-dimensional vector, as in standard PCA, to that of projecting onto some general subspace as follows.

\begin{definition}[Orthogonal Projection]
	Let $V$ be any subspace of $\mathbb{C}^n$, and let $x$ be any signal in $\mathbb{C}^n$.  When $V$ is $p$-dimensional (necessarily $p \leq n$) and equipped with an orthonormal basis $\{ b_1, \ldots, b_p \}$, then the orthogonal projection of $x$ on $V$ is obtained as
	\[
	\proj_V(x) = \sum_{i=1}^p \langle x,b_i \rangle b_i.
	\]
\end{definition}
It is easily verified that the orthogonal project of $x$ on $V$ has energy $\| \proj_V(x) \|^2 = \sum_{i=1}^p | \langle x, b_i \rangle |^2$.

\subsection{The QPCA Problem}

\noindent Suppose that we have centred data $\{\tilde{y}_1, \cdots, \tilde{y}_m\} \subseteq \mathbb{C}^{\tilde{n}}$, which we know has a symbol period with an oversampling rate of $s\in\mathbb{Z}$ (i.e., $s$ samples per symbol).  For each $\tilde{y}_i$, we define the \emph{length-$n$ extension} of the data as follows.  First, define $N = \lfloor\tilde{n}/s\rfloor$ and $n = Ns$.  Then define the length-$n$ data vectors $y_i \in \mathbb{C}^n$ by
\[
\phantom{.}y_i(j) = \begin{cases}
	\tilde{y}_i(j) & 0 \le j < \tilde{n}\\
	0 & \tilde{n} \le j < n
\end{cases}.
\]
We are truncating the vectors so that they have a length that is a multiple of $s$, and we call this multiple $N$.  Our new data has length $n = Ns$.  In a signal-processing context, $N \in \mathbb{Z}$ will correspond to the number of symbol periods, and $s \in \mathbb{Z}$ will correspond to the oversampling rate (i.e., the number of samples per symbol).

We wish to generalize PCA from projections on a space spanned by a single vector $q$ to projections on $\lspan^s(q)$, where $q \in Q^n(s)$ is chosen so that the sum of the energy of the projection of the centred data vectors $\{ y_1, \ldots, y_m \} \subseteq \mathbb{C}^n$ on $\lspan^s(q)$ is as large as possible.  (Setting $N=1$ recovers standard PCA.)  Accordingly, we define a first \emph{$s$-quasicyclic principal component} of the centred data to be a vector $q^{(1)} \in Q^n(s)$ satisfying
\begin{equation}
q^{(1)} =  \max_{q \in Q^n(s)} \sum_{i=1}^m \| \proj_{\lspan^s(q)}(y_i)\|^2.
\label{eq:qpca}
\end{equation}
We refer to the optimization problem (\ref{eq:qpca}) as quasicyclic principal component analysis (QPCA).  We can also think of~(\ref{eq:qpca}) as maximizing the energy of a signal that can be captured when demodulating the signal using the set of $s$-circular-shifts of pulse $q$.

\section{Solution}
\label{sec:solution}

\subsection{Definitions}

\noindent Our optimization problem will turn out to be more tractable to analyze in the frequency domain.  We therefore recall the definitions of the discrete Fourier transform and some related concepts.

\begin{definition}[Discrete Fourier Transform]
	For a positive integer $n$, let $\omega_n = e^{2 \pi \mathbf{i}/n}$, where $\mathbf{i}$ satisfies $\mathbf{i}^2 = -1$.  The \emph{discrete Fourier transform} (DFT) of order $n$ of $x \in \mathbb{C}^n$ is $\mathcal{F}(x) = \hat{x} \in \mathbb{C}^n$ where
	\[
	\hat{x}(k) = \frac{1}{\sqrt{n}} \sum_{i = 0}^{n - 1} x(i) \omega_n^{- i k}.
	\]
	The DFT is invertible as follows:
	\[
	x(i) = \frac{1}{\sqrt{n}} \sum_{k = 0}^{n - 1} \hat{x}(k) \omega_n^{ik}.
	\]
\end{definition}

When the DFT is defined in this way, it has the following useful property.

\begin{claim}[Unitarity of the DFT]
	\label{claim:DFTisUnitary}
	If 
	$x \stackrel{\mathcal{F}}{\longleftrightarrow} \hat{x}$
	and
	$y \stackrel{\mathcal{F}}{\longleftrightarrow} \hat{y}$,
	then
	$\langle x,y \rangle = \langle \hat{x},\hat{y} \rangle$.
\end{claim}

If we can factor $n$ as $n = Ns$, two useful DFT pairs which are easy to verify are the following:
\begin{align}
\mathcal{F}\left\{\sum_{i = 0}^{N - 1}\delta\circledast e_{is}\right\}(k) &= \sqrt{\frac{N}s}\sum_{\ell = 0}^{s - 1}\delta(k - \ell N)\label{eq:picket}\\
\phantom{.}\mathcal{F}\left\{\delta\right\}(k) &= \frac1{\sqrt{n}}\label{eq:del_dft}.
\end{align}

We will also be interested in the DFT of $R_x$, as follows:

\begin{definition}[Energy Spectrum]
	The \emph{energy spectrum} of $x$, denoted as $S_x$, is a scaled version of 	the DFT of $R_x$, i.e.,
	\[
	S_x = \frac{1}{\sqrt{n}} \hat{R}_x; \text{ thus }
	S_x(k)  =  \frac{1}{n} \sum_{j = 0}^{n - 1} R_x(j) \omega_n^{-jk}.
	\]
\end{definition}

These definitions allow us to state the following useful result.

\begin{claim} 
	If $n = Ns$, a unit-energy signal $x$ is $s$-quasicyclic shift-orthogonal if and only if
	\[
	\sum_{j = 0}^{s - 1} S_x(k-Nj) =
	\frac1N
	\] for $0 \le k < N$.
\end{claim}

\begin{proof}
Using~(\ref{eq:rx}), we know that $x$ is quasi-cyclic shift-orthonormal if and only if $R_x \sum_{i = 0}^{N - 1}\delta\circledast e_{is} = \delta$. 
Taking a Fourier transform on each side and applying~(\ref{eq:picket}) and~(\ref{eq:del_dft}) gives the desired result.
\end{proof}

Note that this claim is a frequency-domain version of Nyquist's criterion.

Since $(S_x \circledast \delta)(ks)$ is equal to the total energy of the components of $S_x$ in the set $k + \ideal{N} = \{k, k + N, k + 2N, \dots\}$ this means that a signal $x$ of unit energy is $s$-quasicyclic shift-orthogonal if and only if this total energy is $1/N$ for every $k$.  Stated more precisely, the signal $\hat{x}$ has the orthogonal decomposition 
\begin{equation}\label{eq:orthdecomp}
	\hat{x} = \sum_{t = 0}^{N - 1} \proj_{\mathbb{C}^{t + \ideal{N}}}(\hat{x})
	= \sum_{t = 0}^{N - 1} \hat{x}_{t + \ideal{N}}
\end{equation}
where $\mathbb{C}^{t + \ideal{N}}$ is the $s$-dimensional subspace of $\mathbb{C}^n$ corresponding to the indices in the set $t + \ideal{N}$, and $\hat{x}_{t + \ideal{N}}$ is the projection of $\hat{x}$ onto this space.

This leads to the following useful corollary:

\begin{corollary}
\label{corr:constcoset}
The signal $x$ is $s$-quasicyclic shift-orthonormal if and only if every component of the orthogonal decomposition in~(\ref{eq:orthdecomp}) has the same energy $1/N$, i.e., if and only if
	\begin{equation}
		\| \hat{x}_{t + \ideal{N}} \|^2 = \frac1{N}
	\end{equation}
for $0 \le t < N$.
\end{corollary}

\subsection{Derivation of QPCA}

\noindent We may now reformulate QPCA by expressing the projection onto the span in~(\ref{eq:qpca}) in terms of inner products with the $N$ circular shifts of $q$.
\begin{align}
q^{(1)} &= \arg \max_{q \in Q^{n}(s)} \sum_{i=1}^m
\sum_{j = 0}^{N - 1}| \langle y_i, q \circledast e_{js} \rangle|^2 \\
        &= \arg \max_{q \in Q^{n}(s)} \sum_{i=1}^m
\sum_{j = 0}^{N - 1}| \langle y_i\circledast{e_{js}},q \rangle|^2. 
\label{eq:qpca_prod}
\end{align}
Note that this formulation of QPCA is similar in nature to a standard PCA problem but with respect to an augmented data set $\bigcup_{i=1}^m \bigcup_{j = 0}^{N - 1} \{ y_i \circledast e_{js} \}$ instead of only the $y_i$s.  However, rather than allowing a principal component to be an \emph{arbitrary} signal of unit norm, in this new problem $q^{(1)}$ is required to be $s$-quasicyclic shift-orthonormal, i.e., an element of $Q^n(s)$.

In light of Claim~\ref{claim:DFTisUnitary}, and since elements of $Q^n(s)$ have a convenient energy spectrum characterization, we now pose QPCA in the DFT domain.  Denote by $\hat{Q}^n(s)$ the set of DFTs of elements of $Q^n(s)$, i.e., $\hat{Q}^n(s) = \{ \hat{q} = \mathcal{F}\{q\} : q \in Q^n(s) \}$.  Taking the DFT of each of the centered data vectors, we obtain the sequence $\{ \hat{y}_1, \ldots, \hat{y}_m\}$.  Letting $\hat{e}_{js} = \mathcal{F}\{e_{js}\}$ we seek a signal $\hat{q}^{(1)} \in \hat{Q}^n(s)$ satisfying
\begin{align}
\hat{q}^{(1)} & = \arg \max_{\hat{q} \in \hat{Q}^n(s)}
\sum_{i=1}^m \sum_{j = 0}^{N - 1}
| \langle \sqrt{n}\hat{y}_i\hat{e}_{js},\hat{q} \rangle |^2
\nonumber \\
&= 
\arg \max_{\hat{q} \in \hat{Q}^n(s)}
\sum_{i=1}^m \sum_{j = 0}^{N - 1}
\| \proj_{\hat{q}}( \sqrt{n}\hat{y}_i\hat{e}_{js}) \|^2.
\label{eq:dftqpca}
\end{align}
Again, this problem is similar in nature to a standard PCA problem with an augmented data set $\bigcup_{i=1}^m \bigcup_{j = 0}^{N - 1} \{ \sqrt{n}\hat{y}_i\}$, except that $\hat{q}^{(1)}$ is required to be an element of $\hat{Q}^n(s)$, rather than any vector, as in~(\ref{eq:pca}).

Now we use~(\ref{eq:orthdecomp}) to write our expression in terms of the orthogonal decomposition
$\hat{q} = \sum_{t = 0}^{N - 1} \hat{q}_{t + \ideal{N}}$.  The projection $\proj_{\hat{q}}(z)$ of an arbitrary signal $z \in \mathbb{C}^n$ on $\lspan(\{ \hat{q} \})$ then has the orthogonal decomposition
\[
\proj_{\hat{q}}(z) = \sum_{t = 0}^{N - 1} \proj_{\hat{q}_{t + \ideal{N}}}(z).
\]
Recall from Corollary~\ref{corr:constcoset} that $\hat{q} \in \hat{Q}^{n}(s)$ if and only if $\| \hat{q}_{t + \ideal{N}} \|^2 = \frac{1}{N}$ for all $t$.  Thus, assuming that $\hat{q} \in \hat{Q}^{n}(s)$, we have
\begin{align*}
\proj_{\hat{q}_{t + \ideal{N}}}(z)
&= \frac{ \langle z, \hat{q}_{t + \ideal{N}} \rangle}{
\langle \hat{q}_{t + \ideal{N}}, \hat{q}_{t + \ideal{N}}
\rangle} \hat{q}_{t + \ideal{N}}\\
&= N \cdot \langle z, \hat{q}_{t + \ideal{N}}
\rangle \hat{q}_{t + \ideal{N}},
\end{align*}
which has energy
\begin{align*}
\| \proj_{\hat{q}_{t + \ideal{N}}}(z) \|^2 &= 
N^2 \cdot | \langle z, \hat{q}_{t + \ideal{N}} \rangle|^2
\cdot
\| \hat{q}_{t + \ideal{N}} \|^2\\
&= N \cdot | \langle z, \hat{q}_{t + \ideal{N}} \rangle|^2.
\end{align*}
Thus, for $\hat{q} \in \hat{Q}^{n}(s)$, we get
\begin{align*}
\|\proj_{\hat{q}}(z)\|^2 &=
N \sum_{t  = 0}^{N - 1} |
\langle z, \hat{q}_{t + \ideal{N}} \rangle|^2\\
&= N \sum_{t = 0}^{N - 1} |
\langle z_{t + \ideal{N}}, \hat{q}_{t + \ideal{N}} \rangle|^2.
\end{align*}
The objective function in the optimization problem (\ref{eq:dftqpca}) thus decomposes as
\[
N \sum_{t = 0}^{N - 1} \left (
\sum_{i=1}^m \sum_{j = 0}^{N - 1}
| \langle
( \sqrt{n}\hat{y}_i\hat{e}_{js})_{t + \ideal{N}},
\hat{q}_{t + \ideal{N}} \rangle |^2 \right),
\]
Terms corresponding to a fixed $t$ can be maximized independently, subject only to the constraint that $\| \hat{q}_{t + \ideal{N}} \|^2 = 1/N$ for each $t$.   Maximization of the term corresponding to $t$ is a standard PCA problem in the subspace $\mathbb{C}^{t + \ideal{N}}$, posed as
\begin{align}\label{eq:qhat}
\hat{q}^{(1)}_{t + \ideal{N}}
&= \arg\max_{ \substack{ \hat{q} \in \mathbb{C}^{t + \ideal{N}}\\
\| \hat{q} \|^2 = \frac1N}}
\sum_{i=1}^m \sum_{j = 0}^{N - 1}
| \langle
( \sqrt{n}\hat{y}_i\hat{e}_{js})_{t + \ideal{N}},
q \rangle |^2 \\
&= \frac1{\sqrt{N}}\arg\max_{ \substack{ \hat{q} \in \mathbb{C}^{t + \ideal{N}}\\
		\| \hat{q} \|^2 = 1}}
\sum_{k = 0}^{mN - 1}
| \langle
(\hat{z}_k)_{t + \ideal{N}},
\hat{q} \rangle |^2 \label{eq:qhat_z}
\end{align}
where we have introduced the \emph{augmented data} $\hat{z}_{sj + (i - 1)} = \hat{y}_i\hat{e}_{js}$ to better emphasize the decomposition into regular PCA problems.  An overall QPCA solution in the DFT domain is then obtained as
\begin{equation}\label{eq:qhatFull}
\hat{q}^{(1)} = \sum_{t = 0}^{N - 1} 
\hat{q}^{(1)}_{t + \ideal{N}},    
\end{equation}
from which time-domain solution is obtained as $q^{(1)} = \mathcal{F}^{-1}\left(\hat{q}^{(1)}\right)$.  That is, QPCA on $m$ data vectors of length $n$ reduces to $N$ instances of PCA on $Nm$ data vectors of length~$s$.

\subsection{Non-integer Oversampling Rate}
\label{sec:resample}

\noindent So far, we have assumed that our oversampling rate $s$ is integral, i.e., that we have an integer number of samples per symbol period.  In practice, this need not be the case.  Consider, for instance, the case where the data vectors $y_i$ were sampled from continuous functions with symbol period $T$ with sampling interval $1/\tilde{s}$ that does not divide $T$.

In this case, we proceed as follows: We assume the rate $\tilde{s}$ obeys the Nyquist criterion: in practise, its value may in fact be estimated from the bandwidth of the data, as described above.  Using sinc interpolation, we can construct the following band-limited continuous functions:
\[
\phantom{.}Y_i(t) = \sum_{j = 0}^{n - 1} y_i(j)\sinc(\tilde{s}t - n).
\]
These $Y_i$s can now be sampled at any desired rate $s$ that obeys the Nyquist criterion, giving us new data vectors defined by
\[
y_i^\prime(j) = Y_i(j/s)
\]
for $0 \le j \le \lfloor n/\tilde{s}\rfloor s$ (to correspond with the same total time interval).  If we choose integral $s$, we can then perform QPCA on the $y_i^\prime$s as usual.

\subsection{Lack of Uniqueness}

\noindent It is worth noting that the solution to a PCA problem is not unique.  Since PCA involves maximizing a sum of $|\langle y_i, w\rangle|^2$s where $\|w\| = 1$, it follows that if $w$ is an optimum, so too is $w^\prime = e^{\mathbf{i}\phi}w$ for any real $\phi$.  (In the case of real $w$, this is equivalent to $\pm w$.)  Since QPCA reduces to $N$ instances of regular PCA, this gives us $N$ degrees of freedom, one for each $\hat{q}_{t + \ideal{N}}$.  In particular, if we have a solution in the frequency domain that decomposes as $\hat{q}^{(1)} = \sum_{t = 0}^{N - 1} \hat{q}^{(1)}_{t + \ideal{N}}$, as per~(\ref{eq:qhatFull}), any vector of the form
\begin{equation}\label{eq:nonunique}
\sum_{t = 0}^{N - 1} e^{\mathbf{i}\phi_t}\hat{q}^{(1)}_{t + \ideal{N}}
\end{equation} 
will also be a solution, for any choice of the $N$ real values~$\phi_t$.

Selecting the values of the $\phi_t$s depends on the specific application for QPCA.  One option, consistent with many pre-existing eigendecomposition software routines (e.g., Julia's \texttt{LinearAlgebra: eigen}) is to select the value of $\phi$ that makes the first non-zero element of $w^\prime$ a positive real number.  Another option is to choose the values that make $\hat{q}^{(1)}$ as localized as possible in the time domain.  We will examine some of these options in the examples below.

\subsection{Algorithmic Description}

\noindent The QPCA algorithm is given in Algorithm~\ref{alg:qpca}.  The algorithm proceeds as follows:  First, we calculate the augmented data $\hat{z}$ in terms of the input data.  Next, we perform PCA on each of the appropriate subsets of $\hat{z}$, as in~(\ref{eq:qhat_z}).  Here, $\mathsf{PCA}(x)$ denotes the first standard PCA component of the data $x$.  Finally, we reassemble these results as in~(\ref{eq:qhatFull}) and transform back to the time domain.

\begin{algorithm2e}
\caption{The QPCA Algorithm.}\label{alg:qpca}
\KwData{$N, s, m \in \mathbb{N}$, $x_1, \dots, x_m \in \mathbb{C}^{Ns}$}
\KwResult{$q^{(1)} \in \mathbb{C}^{Ns}$}
\tcc{Calculate DFTs of centred data.}
$\bar{x} \gets \frac1m\sum_{i = 1}^m x_i$\\
\For{$i \gets 1$ \KwTo $m$}{
$\hat{y}_i \gets \mathcal{F}(x_i - \bar{x})$
}
\tcc{Assemble data for PCAs}
\For{$i \gets 1$ \KwTo $m$}{
    \For{$j \gets 0$ \KwTo $N - 1$}{
        \For{$k \gets 0$ \KwTo $Ns - 1$}{
            $\hat{z}[i + mj][k + 1] \gets e^{-2\pi\mathbf{i}jk/N}\hat{y}_i[k + 1]$
        }
    }
}
\tcc{Perform PCAs}
\For{$t \gets 0$ \KwTo $N - 1$}{
    $\hat{z}^t \gets \hat{z}[:][t + 1, t + N + 1, t + 2N + 1, \dots]$\\
    $\hat{q}[t + 1, t + N + 1, t + 2N + 1, \dots] = \sqrt{N}^{-1}\mathsf{PCA}(\hat{z}^t)$\\
}
\KwRet{$\mathcal{F}^{-1}(\hat{q})$}
\end{algorithm2e}


The computational complexity of calculating the first PCA component in an $n\times m$ problem is $O(nm)$, using the power iteration algorithm for eigendecomposition, and assuming it converges quickly (i.e., that the largest-magnitude eigenvalue dominates).  We perform PCA $N$ times on $Nm\times s$ data, for a complexity of $O(N^2m)$, as $s$ is constant.  These factors dominate the complexities of all other operations (e.g., calculating DFTs).

\section{Examples and Applications}
\label{sec:apps}

\noindent In this section, we focus on three applications of QPCA in communications engineering.  In each case, we send $m$ data vectors of $N$ symbols, discretized with an oversampling rate of $s$ samples per symbols.  Our examples make use of our Julia implementation of QPCA, which we have made available on GitHub~\cite{rumsey24}.

\subsection{Recovering a Dominant Pulse}

\noindent In our first example, the data consists of sums of pairs of signals output from a modulation system.  In each continuous-time system, 16QAM data is modulated using root-raised-cosine (RRC) pulses $\psi_\alpha$.  Such pulses have the following continuous-time-domain form for $\tau\in\mathbb{R}$:
\[
\phantom{.}\psi_\alpha(\tau) = \frac{\sin\left[\pi\tau(1 - \alpha)\right] + 4\alpha\tau\cos\left[\pi\tau(1 + \alpha)\right]}{\pi\tau\left[1 - \left(4\alpha\tau\right)^2\right]}.
\]
The parameter $\alpha$ is known as the \emph{roll-off factor}.  This expression is defined by continuity at $\tau = 0$ and $\tau = \pm 1/(4\alpha)$, i.e.,
\begin{align*}
	\psi_\alpha(0) &= 1 + \alpha\left(\frac4\pi - 1\right)\\
	\phantom{.}\psi_\alpha\left(\pm\frac1{4\alpha}\right) &= \alpha\left[\sin\left(\frac{\pi(1 + \alpha)}{4\alpha}\right) - \frac2\pi\cos\left(\frac{\pi(1 + \alpha)}{4\alpha}\right)\right].
\end{align*}  We choose the RRC pulse because the convolution of two of these pulses (e.g., when matched filtering) is shift orthonormal under integer shifts (i.e., with a unit symbol period).

Our received signals are of the form
\begin{multline}\label{eq:mix}
Y_i(\tau) = \sqrt{P_1}\sum_{j = 1}^{N}a_i(j)\psi_{\alpha_1}(\tau - i) +\\ \sqrt{P_2}\sum_{j = 1}^{N}b_i(j)\psi_{\alpha_2}(\tau - i - T) + w(\tau)
\end{multline}
where each of the $a_i$s and $b_i$s is a vector of $N$ QAM symbols in each of the two systems (chosen uniformly and independently, and normalized so that the constellation has average unit power).  We add white Gaussian noise to each sample, realized as $w(\tau)$.  Each system has a different RRC roll-off factor $\alpha_j$, a different average symbol power $P_j$ and $N$ symbols.  The systems have the same symbol period, but are offset from each other by a time-shift~$T$.

We discretize these pulses at oversampling rate $s$ to obtain data vectors $y_i$, defined by $y_i(j) = Y_i(j/s)$ for $j \in\{0, 1, \dots, n\}$, where $n = Ns$.
We apply QPCA to a dataset of $m = 100$ of these vectors for which $\alpha_1 = 0.04$, $\alpha_2 = 0.9$, and $N = 81$.  Our discrete time axis has an oversampling ratio of $s = 9$ samples per symbol period, for a total of $n = 729$ samples per signal.  (That is, our time axis is $\{0, 1/s, 2/s, \dots, \}$.)  A shift of $T = 5$ samples is added to the second system.  In this setup, we examine a variety of $(P_1, P_2)$ pairs, as shown in Figure~\ref{fig:ex1}.

We expect the first QPCA components to be dominated by the pulse that defines the higher-power term in~(\ref{eq:mix}).  This is reflected in the example output spectra shown in Figure~\ref{fig:q1}.  In the $(P_1, P_2) = (0, 1)$ case, we are only sending data with the high roll-off factor $\alpha_2 = 0.9$.  The corresponding pulse has a near-rectangular spectrum, which we see in the corresponding plot of $|\hat{q}^{(1)}|$.  Similarly, when $(P_1, P_2) = (1, 0)$, we are only sending data with the low roll-off factor $\alpha_1 = 0.04$.  Here, the pulse spectrum is close to a lobe of a cosine, which we also see in the figure.  The intermediate $(P_1, P_2)$ pairs correspond to $\hat{q}^{(1)}$ pulses lying between these extremes.

In Figure~\ref{fig:q2}, we see examples of $|\hat{q}^{(2)}|$.  For the extreme cases $(P_1, P_2) \in \{(0, 1), (1, 0)\}$, the secondary pulses  $q^{(2)}$ consist only of out-of-band noise, since the primary pulse $q^{(1)}$ capture all of the data.  The other cases give intermediate results, with the second components having greater magnitudes at frequencies that are more represented in the non-dominant pulse.

\begin{figure}[!t]
    \centering
    \begin{subfigure}[b]{0.4\textwidth}
    	\centering
    	\includegraphics[width=\linewidth]{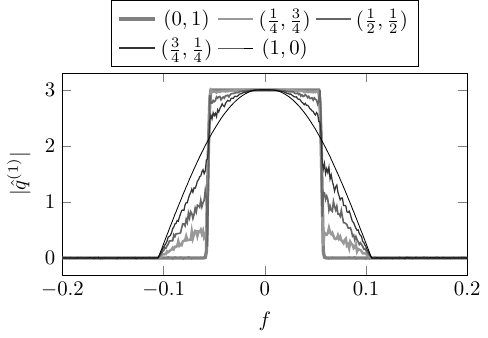}
    	\caption{First QPCA component.}
    	\label{fig:q1}
    \end{subfigure}
	\vspace*{20pt}
	
    \begin{subfigure}[b]{0.4\textwidth}
    	\centering
    	\includegraphics[width=\linewidth]{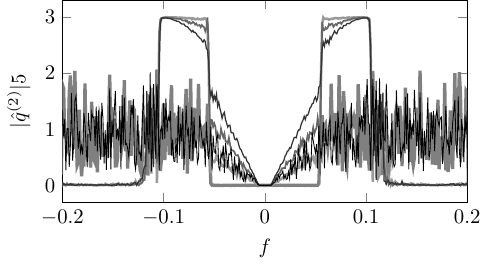}
    	\caption{Second QPCA component}
    	\label{fig:q2}
    \end{subfigure}
    \caption{Absolute values of output spectra of first two QPCA components for a variety of mixtures of pulses, listed as $(P_1, P_2)$.}
    \label{fig:ex1}
\end{figure}

Figure~\ref{fig:ex1} does not show the phase of the results, which are only defined up to the phase ambiguity in~(\ref{eq:nonunique}).  In the case of this example, the underlying pulses RRC are real even functions of time.  This means that their Fourier transforms are real even functions of frequency.  The implication is that our spectra with absolute values (i.e., with all phases zero) are in fact solutions.

\subsection{Symbol Period Estimation}

\noindent In a practical setting, the QPCA user may not know the symbol period (i.e., $s$) \emph{a priori}.  In this section, we investigate some approaches to determine this period.

What constitutes a good choice of $s$ may depend on the application, but a reasonable metric is to choose an $s$ that maximizes the ratio of the energies in the first two QPCA components, i.e., the ratio of the dominant eigenvalues of the underlying PCA problems.  This is a good choice as it corresponds to the case where the signal energy is most concentrated in one component family.

To illustrate the performance of this metric, we send 16QAM symbols with a setup similar to that of the previous example: RRC pulses with a rolloff factor of $\alpha = 0.5$, $N = 100$ symbols per run, and $m = 100$ total runs.  The transmitter has an oversampling rate of $s=9$ samples per symbol.  None of $\alpha$, $N$, and $s$ is known to the receiver.  Our symbols are chosen to have unit power on average, and we add white Gaussian noise with variance $\sigma=0.1$.

We consider a range of values of $s$, truncating the data to multiples of $s$ as necessary.  For each candidate value of $s$, we perform QPCA, and examine the resulting principal component $q^{(1)}$, as well as the fraction $\lambda_1$ of the total energy captured by $q^{(1)}$ and its family of $s$-shifts.  We repeat for $q^{(2)}$, whose family  captures fraction $\lambda_2$, in order to obtain our metric $\lambda_1/\lambda_2$.  Figure~\ref{fig:ex2eig} shows this ratio for a range of values of $s$.  Figure~\ref{fig:ex2pulses} shows some examples of $|q^{(1)}|$, corresponding to different values of~$s$.

\begin{figure}[!t]
	\centering
	\includegraphics[width=\linewidth]{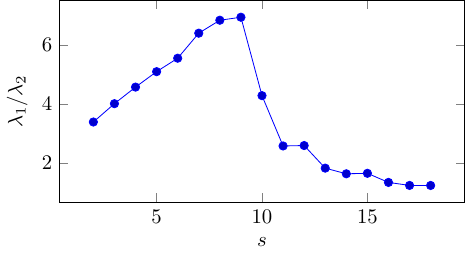}
	\caption{Ratio of first to second eigenvalues for each candidate value of $s$.}
	\label{fig:ex2eig}
\end{figure}

\begin{figure}[ht]
	\centering
	\includegraphics[width=\linewidth]{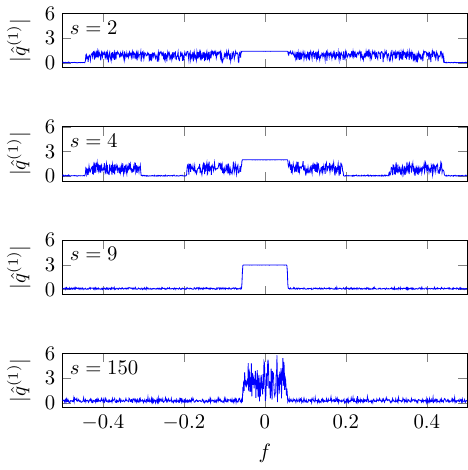}
	\caption{Primary pulses returned by QPCA for different values of $s$ in Example~2.}
	\label{fig:ex2pulses}
\end{figure}

These figures illustrate several points.  First, we see in Figure~\ref{fig:ex2eig} that the ratio of eigenvalues peaks at the correct value of $s=9$.  This is where we have the best tradeoff between the first pulse capturing as much signal and as little noise as possible.  This falls off quickly as $s$ increases, with the energy being distributed among an increasing number of component families.  Second, we see in Figure~\ref{fig:ex2pulses} that the support of the pulses is larger for $s < 9$.  As $s$ increases beyond this point, the pulses change shape, but the support remains the same.

If, as will often be the case, the underlying pulse is band-limited, the value of $s$ can also be estimated from the bandwidth of the signal.  For instance, in this case, we can estimate the signal bandwidth at 1/9 the total bandwidth (in the absence of noise, this estimation is exact).  This means that we should concentrate our search near $s = 9$.

\subsection{Recomputing the Oversampling Factor}

\noindent Our final example concerns data which has been sampled at a rate that does not divide the symbol period.  In particular, we look at signals of the form
\[
Y(\tau) = \sum_{i = 1}^N a_i\psi_{\alpha}(\tau - i)
\]
sampled as $y(j) = Y(j/s)$ for a non-integer $s$.

In our experiment, we take $\alpha = 0$, $N = 100$, and $s=8.5$.  This means that our $y_i$s have length $\lfloor sN\rfloor = 850$.  Examining the bandwidths of the signals suggests using QPCA with integer $s$ chosen to be 8 or 9. 

\begin{figure}[!t]
	\centering
	\begin{subfigure}[b]{0.4\textwidth}
		\centering
		\includegraphics[width=\textwidth]{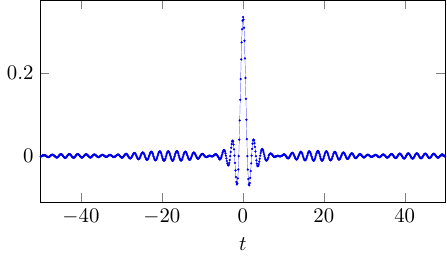}
		\caption{Pulse $q^{(1)}$ obtained from RRC data sampled at 8.5 samples per symbol period.  Here, $\lambda_1 = 94.6\%$.}
		\label{fig:badS}
	\end{subfigure}
	\hfill
	\begin{subfigure}[b]{0.4\textwidth}
		\centering
		\includegraphics[width=\textwidth]{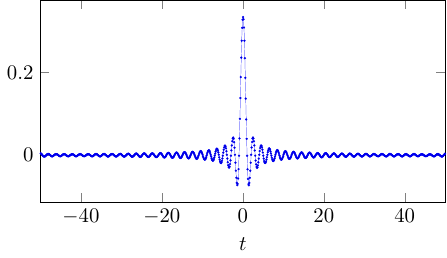}
		\caption{Pulse $q^{(1)}$ obtained from same RRC data resampled at 9 samples per symbol period.  Here, $\lambda_1 = 99.1\%$.}
		\label{fig:goodS}
	\end{subfigure}
	\caption{Comparison of pulses obtained when sampling RRC data with a non-integer oversampling rate, and the same data resampled at an appropriate integer rate.}
	\label{fig:resample}
\end{figure} 

If we perform QPCA on this data with $s=9$, we obtain a first component pulse $q^{(1)}$ shown in Figure~\ref{fig:badS}.  This is a poor approximation of the underlying RRC pulse, and the QPCA analysis tells us that this pulse and its shifts capture only $\lambda=94.6\%$ of the available energy.

On the other hand, we can resample the data using the method described in Section~\ref{sec:resample} so that each symbol period now includes an integer number of samples.  We take $s = 9$, which is the next largest integer (which we can estimate from the signal bandwidth).  In this case, the first component pulse $q^5{(1)}$ shown in Figure~\ref{fig:goodS}.  This pulse has a much more RRC pulse-like shape, and its family captures $\lambda_1 = 99.1\%$ of the available energy.  (The lost 0.9\% is the result of the combination of randomness in the choice of symbols and truncation effects from the pulse.)  This shows the importance and effectiveness of resampling data which was not sampled at an integer oversampling rate.

\section{Conclusion and Future Work}

\noindent We have developed quasi-cyclic principal component analysis (QPCA), a statistical tool for decomposing data into a representation formed as a linear combination of shift-orthogonal signals.  We have also provided three examples that illustrate the use of this tool.  Future work on this algorithm may include extending QPCA to other spaces: one possibility is to examine data vectors taken from two dimensional spaces, which has the potential to be useful in, for example, identifying repeating tiling patterns when image processing.

\bibliographystyle{IEEEtran}
\bibliography{references}

\end{document}